\documentclass[a4paper,11pt]{article}

\usepackage[utf8]{inputenc} 
\usepackage{amsfonts}
\usepackage{amssymb}
\usepackage{amsthm}
\usepackage{amsmath}
\usepackage{a4wide}
\usepackage{mathrsfs}
\usepackage{epsfig}
\usepackage{comment}
\usepackage{palatino}
\usepackage{tikz}
\usepackage{fp,ifthen}
\usepackage{esint}
\usepackage{nicefrac}
\usetikzlibrary{decorations.pathreplacing}
\usepackage{float}
\usepackage{enumerate} 
\usepackage{enumitem} 

\usepackage[a4paper,twoside,top=2.5cm, bottom=2cm, left=2.3cm, right=2.3cm]{geometry} 

 \usepackage{lipsum}

\usepackage[colorinlistoftodos,textwidth=2.3cm]{todonotes} 

\newcommand{\pdfgraphics}{\ifpdf\DeclareGraphicsExtensions{.pdf,.jpg}\else\fi}
\usepackage{graphicx}

\usepackage{color}
\definecolor{hanblue}{rgb}{0.27, 0.42, 0.81}
\definecolor{red}{rgb}{1.0, 0.0, 0.0}
\usepackage[colorlinks, citecolor=blue,linkcolor=blue, urlcolor = blue]{hyperref}
\usepackage[english,capitalize]{cleveref}

\usepackage{mathtools}

\theoremstyle{plain}

\newtheorem{theorem}{Theorem}[section]
\newtheorem{lemma}[theorem]{Lemma}
\newtheorem{proposition}[theorem]{Proposition}

\theoremstyle{definition}
\newtheorem{definition}[theorem]{Definition}

\theoremstyle{remark}
\newtheorem{remark}[theorem]{Remark}

\numberwithin{equation}{section}

\newcommand{\de}{\ensuremath{\,\mathrm d}} 

\newcommand \eps{\ensuremath{\varepsilon}} 
\renewcommand{\epsilon}{\varepsilon}
\newcommand{\N}{\ensuremath{\mathbb N}}
\newcommand{\R}{\ensuremath{\mathbb R}}

\makeatletter
\renewcommand*\env@matrix[1][*\c@MaxMatrixCols c]{%
  \hskip -\arraycolsep
  \let\@ifnextchar\new@ifnextchar
  \array{#1}}
\makeatother

\begin{document}

\pdfgraphics 

\title{Sets of finite perimeter on Riemannian manifolds and stochastic completeness}

\author{Luca Gennaioli \footnote{\href{L.Gennaioli@warwick.ac.uk}{L.Gennaioli@warwick.ac.uk}}}

\date{\today}

\maketitle

\vspace{-0.5cm}
\begin{abstract}
We prove a heat semigroup characterization of the total variation for compactly supported ${\rm BV}$ on arbitrary smooth complete weighted Riemannian manifolds, extending the main result in \cite{GP15}. We then provide an example of a weighted manifold where such equivalence does not hold for a large class of sets of finite perimeter.
\end{abstract}

\tableofcontents

\section{Introduction}
\subsection{Overview on the topic}
The theory of sets of finite perimeter in $\R^n$ was formally introduced by De Giorgi in \cite{DG53}. His definition used the \emph{heat kernel}
\begin{equation*}
    p_t(x,y)=\frac{1}{(4\pi t)^\frac{n}{2}}e^{-\frac{|x-y|^2}{4t}}
\end{equation*}
to smooth out the characteristic function of a set $\Omega$, $\chi_\Omega$, and then define the perimeter as follows 
\begin{equation}
\label{def:per_heat}
    {\rm Per}(\Omega)=\lim_{t\to 0^+}\int_{\R^n}|\nabla p_t\ast\chi_\Omega|(x)\de x.
\end{equation}
In the later works \cite{DG54} and \cite{DG55} he then managed to reconcile the theory with the one pioneered by Caccioppoli and with the one of  Carathéodory, Hausdorff and Federer, based on measure theoretic concepts and on rectifiability. These works led to an equivalent definition of perimeter, that is 
\begin{equation}
    \label{def:per_div}
      {\rm Per}(\Omega)=\sup\bigg\{\int_\Omega{\rm div}X\de x:\,X\in C^1_c(\R^n;\R^n),\;|X|\leq 1\bigg\}.
\end{equation}
Replacing $\chi_\Omega$ by a general $L^1_{\rm loc}$ function $u$, De Giorgi gave also the definition of total variation of the gradient of a function, introducing the concept of function bounded variation (BV function) of several variables.
\newline
The theory of BV functions and of sets of finite perimeter has since then found countless applications and has developed in various research directions. One of the core questions revolving around the notion of total variation of the gradient of a BV function (and consequently of perimeter) was if can one define such quantity (and develop a similar theory) on more general ambient spaces. 
\newline
In the case of smooth structures such as Riemannian manifolds several contributions appeared over the years: the first one is probably \cite{MPPP07}, where the authors, using the heat kernel of the Riemannian manifold $(M,g)$, were able to prove the equivalence between \eqref{def:per_heat} and \eqref{def:per_div} assuming a lower bound on the Ricci curvature and a non-collapsing assumption on the volume of the balls. In \cite{CM07} the authors were then able to remove the non-collapsing assumption and gave a simpler proof, exploiting bounds on the heat flow of differential forms. Finally in \cite{GP15} the authors improved on all the previous approaches by removing a uniform lower Ricci bound and asking for a Kato-type condition on the Ricci tensor which, roughly speaking, asks for the negative part of the such tensor to be integrable. The final tasks would be to establish the validity or the non-validity of the equivalence between \eqref{def:per_heat} and \eqref{def:per_div} on general smooth Riemannian manifolds.
\newline
In this note we shall discuss both type of phenomena: on one hand we will show that on any weighted manifold the equivalence holds true for functions $u\in{\rm BV}(M)$ provided the support of $u$ is bounded (see Theorem \ref{thm:bounded_asympt}). The requirement of boundedness of the support is not merely technical: to obtain good "tail estimates" on the heat kernel it seems a necessary condition. It would be nice to understand if this approach can be extended to cover the case of general ${\rm BV}$ functions. On the other hand we will construct an example of a weighted manifold (see Theorem \ref{thm:equiv_failure} and also Remark \ref{rem:counter_manifold} for the same pathology on a conical manifold $(M,g)$) where such equivalence is no longer true if one considers sets of finite perimeter which are the complement of bounded sets. The key in obtaining such counterexample is to exploit stochastic incompleteness of the (weighted) manifold we build along with a stronger property of the heat kernel of this specific example (see Proposition \ref{prop:grad_heat_blowup}).

\subsection{Statement of the main results}
We are now going to state our main results, referring to Section \ref{sec:prelim} for the relevant definitions.

The first theorem we are going to prove establishes that, on general weighted manifolds $(M,g,e^{-V}{\rm vol})$ (we shall simply write $(M,g,\mu)$ for short), the equivalence between \eqref{def:per_heat} and \eqref{def:per_div} is still in place if the ${\rm BV}$ function we are dealing with is compactly supported.
\begin{theorem}
\label{thm:bounded_asympt}
    Let $(M,g,\mu)$ be a smooth, complete, $n$-dimensional weighted Riemannian manifold, with $V\in C^\infty(M)$. Then for any compactly supported $u:M\to\R$ such that $u\in L^1(M)$ we have
    \begin{equation*}
        \lim_{t\to 0^+}\int_M|\nabla {\rm h}_tu|\de\mu=|Du|(M).
    \end{equation*}
\end{theorem}
\begin{remark}
    We would like to point out the assumption on the compactness of the support of $u$. This assumption was not needed in \cite{MPPP07}, \cite{CM07} and \cite{GP15} due to the uniformity on the estimates for the heat kernel. However in our case this turns out to be crucial (in general) to exploit estimates for  the heat kernel.
\end{remark}
\begin{remark}
    If the manifold $(M,g)$ is stochastically complete, then we have ${\rm h}_tu=1-{\rm h}_t(1-u)$. In the case of sets this means that $\nabla{\rm h}_t\chi_E=-\nabla{\rm h}_t\chi_{E^c}$ and if $E\subset M$ is bounded, then 
    \begin{equation*}
        \lim_{t\to 0^+}\int_M|\nabla {\rm h}_t\chi_{E^c}|\de\mu={\rm Per}(E).
    \end{equation*}
    Therefore the sought characterization extends to a class of non-compact sets with infinite measure. An interesting topic to investigate would be if it is possible to remove the assumption of compactness of the supports and ask just for integrability (as the proofs in \cite{MPPP07}, \cite{CM07} and \cite{GP15} do). Furthermore, we highlight that all the assumptions in \cite{MPPP07}, \cite{CM07} and \cite{GP15} entail that $(M,g)$ is stochastically complete.
\end{remark}
The second crucial and perhaps surprising result is that there exists a weighted manifold for which the mentioned equivalence does not hold for \emph{any} characteristic function of the complement of a bounded set of finite perimeter. More precisely, we have the following:
\begin{theorem}
\label{thm:equiv_failure}
    There exists a smooth, complete, weighted manifold $(M,g,\mu)$ which is stochastically incomplete, such that for all bounded sets $E\subset M$ of finite perimeter one has 
    \begin{equation}
    \label{eq:blowup_set}
        \int_M|\nabla {\rm h}_t\chi_{E^c}|\de\mu=\infty\quad\forall t>0.
    \end{equation}
\end{theorem}
\section{Preliminaries}
\label{sec:prelim}
For a weighted manifold $(M,g,\mu)$  
the weighted divergence ${\rm div}$ is defined as the operator acting on smooth vector fields $X$ as 
\begin{equation*}
    {\rm div}X={\rm div}_gX-g(\nabla V, X),
\end{equation*}
where ${\rm div}_g$ is the classical divergence and $\nabla$ is the classical gradient.
Consequently, the weighted Laplacian $\Delta$ is defined on smooth functions $u:M\to\R$ as $\Delta u = {\rm div}\nabla u$, and hence we have
\begin{equation*}
    \Delta u = \Delta_g u-g(\nabla V,\nabla u),
\end{equation*}
where $\Delta_g={\rm div}_g\nabla$ is the classical Laplace-Beltrami operator of the manifold. One can also define a weighted Ricci tensor, given by 
\begin{equation*}
    {\rm Ric}={\rm Ric}_g+{\rm Hess}V,
\end{equation*}
where ${\rm Ric}_g$ is the Ricci tensor of $(M,g)$.
With the latter notions one can write the Bochner formula 
\begin{equation}
\label{eq:weighted_Bochner}
    \Delta\bigg(\frac{|\nabla u|^2}{2}\bigg)=|{\rm Hess}(u)|^2_{\rm HS}+{\rm Ric}(\nabla u,\nabla u)+g(\nabla u,\nabla\Delta u),
\end{equation}
where $|{\rm Hess}(u)|^2_{\rm HS}$ is the squared Hilbert-Schmidt norm of the Hessian of $u$.
We did not highlight the dependence on the weight of the operators to lighten a bit the notation, since such symbols will appear frequently.

The weighted Laplacian $\Delta$ induces the notion of \emph{heat kernel}. For all $t>0$ we shall denote by $p_t:M\times M\to(0,\infty)$ the minimal non-negative fundamental solution of the heat equation (which we call heat kernel). Given a function $u$ which is either in $L^1(M)$ or in $L^\infty(M)$ we define the function ${\rm h}_tu:M\to\R$ as
\begin{equation*}
    {\rm h}_tu(x)=\int_M p_t(x,y)u(y)\de\mu(y).
\end{equation*}
Such function is clearly smooth on $(0,\infty)\times M$ if $u\in L^1(M)$ and satisfies the heat equation pointwise. These properties hold even if $u\in L^\infty(M)$, though since we could not find a reference in the literature we shall briefly discuss this matter: first without loss of generality we shall assume that $u\geq 0$ (otherwise we would just have to argue separately for the positive and the negative part). Then by H\"older inequality we have ${\rm h}_tu\leq\|u\|_\infty$ for all $t>0$, thus proving that ${\rm h}_{\cdot} u\in L^1_{\rm loc}((0,\infty)\times M)$. We can finally infer the claim after an application of \cite[Theorem 7.15]{G09}.

We stress that for a bounded (even continuous) initial datum $u$ without additional integrability solutions to the heat equation are not unique in general. For that to be the case we have to ask for the following property to hold.
\begin{definition}
\label{def:stoc_compl}
    We say that a complete weighted $n$-dimensional manifold $(M,g,\mu)$ is stochastically complete if 
    \begin{equation*}
        \int_Mp_t(x,y)\de\mu(y)=1\quad\forall t>0,\,x\in M.
    \end{equation*}
    Otherwise we call $(M,g,\mu)$ stochastically incomplete.
\end{definition}

Regarding non-uniqueness we refer the reader to \cite[Theorem 8.18]{G09}.
We mention that for ${\rm h}_t u$ the semigroup property holds, that is for all $t,s>0$ one has
\begin{equation*}
    {\rm h}_{t+s}u={\rm h}_t({\rm h}_s u).
\end{equation*}
We finally record the following Lemma which is a combination of \cite[Lemma 2.16]{CFSS24} and \cite[Lemma 2.17]{CFSS24} and provides key estimates for the heat kernel.
\begin{lemma}
\label{lem:small_time_heat}
    Let $(M,g,\mu)$ be a smooth, complete $n$-dimensional weighted Riemannian manifold. Then, for all $p\in M$ and $r>0$ there exists $R>r>0$ and constant $C,c>0$ depending only on $R,r$ and $p$ such that
    \begin{equation*}
        \int_{M\setminus B_R(p)}|\nabla p_t|(x,y)\de\mu(y)\leq Ce^{-\frac{c}{t}},\; |{\rm Hess}p_t|_{\rm HS}(x,z)\leq Ce^{-\frac{c}{t}}
    \end{equation*}
    for all $x\in B_{r}(p)$, $z\in B_R^c(p)$ and $t>0$.
\end{lemma}
\begin{proof}
    The proof follows from \cite[Lemma 2.16]{CFSS24}  and \cite[Lemma 2.17]{CFSS24} after covering $B_r(p)$ with a, crucially finite, number of charts. Their proof is for smooth Riemannian manifolds but the modifications for the weighted case are minimal, we hence skip the details.
\end{proof}
We can now define functions of bounded variation as follows.
\begin{definition}
\label{def:BV}
    We say that a measurable function $u:M\to\R$ belongs to ${\rm BV}(M)$ if $u\in L^1(M)$ and 
    \begin{equation}
    \label{def:tot_var}
        |Du|(M):=\sup\bigg\{\int_Mu\,{\rm div}X\de\mu:\;X\in C^1_c(M;TM),\,|X|\leq 1\bigg\}<+\infty,
    \end{equation}
    where $C_c^1(M;TM)$ denotes denotes the class of compactly supported vector fields of class $C^1$ and $|X|=g(X,X)^\frac{1}{2}$. Moreover we say that a set $\Omega\subseteq M$ has finite perimeter if $\chi_\Omega$ satisfies \eqref{def:tot_var}, without requiring $\Omega$ to be of finite measure.
\end{definition}
\section{Proofs of the main results}

\begin{proof}[Proof of Theorem \ref{thm:bounded_asympt}]
First of all the following inequality is true for any function $u\in L^1(M)$
\begin{equation*}
    |Du|(M)\leq\liminf_{t\to 0^+}\int_M|\nabla{\rm h}_tu|\de\mu,
\end{equation*}
following by lower semicontinuity of the total variation under $L^1(M)$ convergence.
    To prove the $\limsup$ inequality we can of course assume $u\in {\rm BV}(M)$. Let $\varepsilon>0$ be fixed and choose $R\geq r>0$ and $p\in M$ such that the following are satisfied: for such values Lemma \ref{lem:small_time_heat} holds, ${\rm Per}(B_{R}(p))<+\infty$ (this is possible thanks to Coarea formula) and ${\rm supp}u\subset B_r(p)$. Then we claim
    \begin{equation}
    \label{eq:ball_reduction}
        \limsup_{t\to 0^+}\int_{M}|\nabla {\rm h}_tu|\de\mu= \limsup_{t\to 0^+}\int_{B_R(p)}|\nabla {\rm h}_tu|\de\mu.
    \end{equation}
   Indeed there holds
  \begin{equation*}
\int_{B_R^c(p)}|\nabla {\rm h}_tu|\de\mu\leq\int_{B_r(y)}|u|(y)\int_{B_R^c(p)}|\nabla_xp_t|(x,y)\de\mu(x)\de\mu(y)\leq C e^{-\frac{c}{t}}\|u\|_{L^1(M)}, 
  \end{equation*} 
  and the last term goes to zero as $t\to 0^+$ (note that $C$ depends on $R$).
  The proof now follows a similar argument of the one contained in \cite{MPPP07} but with several crucial differences: we are going to differentiate the function 
  \begin{equation*}
      F_\eps(t):=\int_{B_{R}(p)}\sqrt{|\nabla {\rm h}_tu|^2+\eps}\de\mu=\int_{B_R(p)}W_\eps\de\mu.
  \end{equation*}
Observe that 
\begin{equation*}
    \partial_tW_\eps=\frac{g(\partial_t\nabla {\rm h}_tu,\nabla {\rm h}_tu)}{W_\eps}=\frac{g(\nabla\Delta {\rm h}_tu,\nabla {\rm h}_tu)}{W_\eps}.
\end{equation*}
Bochner identity \eqref{eq:weighted_Bochner} then gives 
\begin{equation*}
    \frac{g(\nabla\Delta {\rm h}_tu,\nabla {\rm h}_tu)}{W_\eps}=\frac{1}{2}\frac{\Delta(|\nabla {\rm h}_t u|^2)}{W_\eps}-\frac{{\rm Ric}(\nabla {\rm h}_tu,\nabla {\rm h}_tu)}{W_\eps}-\frac{|{\rm Hess}({\rm h}_tu)|^2_{\rm HS}}{W_\eps}.
\end{equation*}
All in all we get 
\begin{equation*}
    F^\prime_\eps(t)=\frac{1}{2}\int_{B_{R}(p)}\frac{\Delta(|\nabla {\rm h}_t u|^2)}{W_\eps}\de\mu-\int_{B_{R}(p)}\frac{{\rm Ric}(\nabla {\rm h}_tu,\nabla {\rm h}_tu)}{W_\eps}\de\mu-\int_{B_{R}(p)}\frac{|{\rm Hess}({\rm h}_tu)|^2_{\rm HS}}{W_\eps}\de\mu.
\end{equation*}
For what concerns the first term on the r.h.s. we have
\begin{align*}
    \int_{B_{R}(p)}\frac{\Delta(|\nabla {\rm h}_t u|^2)}{W_\eps}\de\mu&=2\int_{\partial B_{R}(p)}{\rm Hess}({\rm h}_tu)\bigg(\frac{\nabla {\rm h}_t u}{|\nabla {\rm h}_t u|},\nu_{\partial B_{R}(p)}\bigg)\de\sigma -\int_{B_{R}(p)}g(\nabla (W_\eps^{-1}),\nabla|\nabla {\rm h}_t u|^2)\de\mu \\
    &=2\int_{\partial B_{R}(p)}{\rm Hess}({\rm h}_tu)\bigg(\frac{\nabla {\rm h}_t u}{W_\eps},\nu_{\partial B_{R}(p)}\bigg)\de\sigma +\frac{1}{2}\int_{B_{R}(p)}\frac{|\nabla|\nabla {\rm h}_t u|^2|^2}{W_\eps^3}\de\mu \\
    &\leq C\|u\|_{L^1(M)} e^{-\frac{c}{t}}{\rm Per}(B_{R}(p))+2\int_{B_{R}(p)}\frac{|{\rm Hess} ({\rm h}_tu)|_{\rm HS}^2}{W_\eps}\de\mu,
\end{align*}
where we used Gauss-Green theorem, Lemma \ref{lem:small_time_heat} and standard inequalities. Finally we can use the fact that there exists $K\in\R$ such that ${\rm Ric}\geq-K$ on $B_R(p)$. This fact, together with the previous inequalities, yields

\begin{equation*}
    F_\eps^\prime(t)\leq K\int_{B_R(p)}\frac{|\nabla {\rm h}_tu|^2}{W_\eps}\de\mu + Ce^{-\frac{c}{t}}\leq  K F_\eps(t)+Ce^{-\frac{c}{t}},
\end{equation*}
for all $t\in(0,1)$. Integrating the latter between $0<s<t<1$ then gives
\begin{equation*}
    F_\eps(t)\leq e^{K(t-s)}F_\eps(s)+Ce^{Kt}\int_s^te^{-Kr-\frac{c}{r}}\de r.
\end{equation*}
Finally, we can pass to the limit as $\eps\to 0^+$ and thanks to dominated convergence we get 
\begin{equation}
\label{eq:Gronwall}
     F(t)\leq e^{K(t-s)}F(s)+Ce^{Kt}\int_s^te^{-Kr-\frac{c}{r}}\de r,
\end{equation}
Where $F(t)=F_0(t)$.
Now we choose a sequence $(u_j)_{j\in\N}\subset C^\infty_c(M)$ such that $u_j\to u$ strictly in ${\rm BV}(M)$. By classical semicontinuity we have 
\begin{equation*}
    \int_{B_R(p)}|\nabla {\rm h_{t}}u|\de\mu\leq\liminf_{j\to\infty}\int_{B_R(p)}|\nabla {\rm h_{t}}u_j|\de\mu,
\end{equation*}
so that we can combine \eqref{eq:ball_reduction} and \eqref{eq:Gronwall} to get
\begin{align*}
    \limsup_{t\to 0^+}\int_{B_R(p)}|\nabla {\rm h_{t}}u|\de\mu&\leq  \limsup_{t\to 0^+}\liminf_{j\to\infty}\int_{B_R(p)}|\nabla {\rm h_{t}}u_j|\de\mu \\
    &\leq\limsup_{t\to 0^+}\liminf_{j\to\infty}\bigg[e^{Kt}|Du_j|(B_R(p))+Ce^{Kt}\int_0^te^{-Kr-\frac{c}{r}}\de r\bigg] \\
    &=\limsup_{t\to 0^+}\bigg[e^{Kt}|Du|(M)+Ce^{Kt}\int_0^te^{-Kr-\frac{c}{r}}\de r\bigg] \\
    &=|Du|(M).
\end{align*}
The previous inequality concludes the proof.
\end{proof}

The next Proposition, building a pathological manifold, is the key for the proof of Theorem \ref{thm:equiv_failure} and we believe it is of independent interest.

\begin{proposition}
\label{prop:grad_heat_blowup}
    There exists a smooth weighted manifold $(M,g,\mu)$ which is complete and such that
    \begin{equation}
        \label{eq:grad_heat_blowup}\int_M\bigg|\nabla_x\int_Mp_t(x,y)\de\mu(y)\bigg|\de\mu(x)=+\infty\quad\forall t>0.
    \end{equation}
\end{proposition}
\begin{proof}
    Consider $\R^3$ with the Euclidean metric and $\mu=e^{|\cdot|^4}\mathscr{L}^3$.
Now we define the quantity 
\begin{equation}
\label{eq: radial_m}
    m(t,x):=\int_M p_t(x,y)\de\mu(y),
\end{equation}
which is such that $0\leq m\leq 1$. Note that $m$ is radial in the space variable and that $t\mapsto m(t,\rho)$ (here we keep denoting $m$ the radial counterpart of the function defined in \eqref{eq: radial_m}) is non increasing since $m(t+s,\cdot)={\rm h}_t(m(s,\cdot))\leq {\rm h}_t1=m(t,\cdot)$. Now set $u(t,r)=m(t,|x|)$, where $r=|x|$. Thanks the explicit expression of the weighted Laplacian, it is easy to verify that $u$ solves the following
\begin{equation}
\label{eq:1d_heat}
    \partial_t u(t,r)=\frac{1}{A(r)}\partial_r\big(A(r)\partial_ru(t,r)\big),
\end{equation}
where $A(r)=r^2e^{r^4}$. 

\emph{Claim:} $\lim_{r\to\infty}u(t,r)=0$ for all $t>0$.

To prove such claim let $u_R={\rm h}_t^{B_R}1$ be the solution to the problem 
\begin{equation*}
   \begin{cases}
    \partial z = \Delta z\quad{\rm in}\;(0,\infty)\times{B_R}\\
    z=0\quad{\rm on}\;(0,\infty)\times\partial B_R \\ 
    z(0,\cdot)=1\quad{\rm on}\;B_R;
\end{cases} 
\end{equation*}

Then by \cite[Exercise 7.40]{G09} we have that for all $t,r$ fixed $u_R(t,r)\to u(t,r)$ as $R\to\infty$ monotonically from below. Now set 
\begin{equation*}
    v_R(t,r)=\int_0^tu_R(s,r)\de s,\quad w_R(r)=\int_r^R\frac{1-e^{-s^4}}{s^3}\de s.
\end{equation*}
Observe that (abusing a bit the notation) for all $t>0$ we get
\begin{equation*}
    \Delta v_R(t,\cdot)=\int_0^t\partial_su_R(s,\cdot)\de s=u_R(t,\cdot)-1\geq -1,\quad v_R=0\quad{\rm on}\;\partial B_R
\end{equation*}
and also
\begin{equation*}
    \Delta w_R=-4+\frac{e^{r^4}-1}{r^4e^{r^4}}< -1\quad{\rm on}\,B_R,\quad w_R=0\quad{\rm on}\;\partial B_R.
\end{equation*}
Together the previous facts yield
\begin{equation*}
    \Delta(v_R-w_R)\geq 0\quad{\rm on}\,B_R,\quad v_R-w_R=0\quad{\rm on}\;\partial B_R,
\end{equation*}
for all $t>0$.
By the maximum principle we therefore infer $v_R\leq w_R$ on $B_R$ for all $t>0$. Since (by the usual semigroup identity) $t \mapsto u_R(t,\cdot)$ is non-increasing we get
\begin{equation*}
    tu_R(t,r)\leq\int_0^tu_R(s,r)\de s=v_R(t,r)\leq w_R(r).
\end{equation*}
Therefore letting first $R\to\infty$ and then $r\to\infty$ we infer the claim thanks to the equiintegrability of $w_R$. Observe that the claim also implies that the manifold $(M,g,\mu)$ is stochastically incomplete. Now again by the maximum principle and the monotone convergence of the Dirichlet heat flow we get $u(t,0)\geq u_R(t,0)>0$ for every $t>0$. We now integrate \eqref{eq:1d_heat} using the smoothness at $r=0$ so we end up with
\begin{equation*}
    A(r)\partial_ru(t,r)=\int_0^rA(s)\partial_tu(t,s)\de s\leq 0.
\end{equation*}
The latter means $\partial_ru(t,r)\leq 0$, due to the fact that $A\geq 0$ and $\partial_tu\leq 0$. Now since $u(t,0)>0$, $\partial_ru\leq 0$ and $\lim_{r\to\infty}u(t,r)=0$, for all $t>0$ there exists $r_t$ such that $\partial_ru(t,r_t)<0$. We then have
\begin{equation*}
    q_t(r):=-A(r)\partial_ru(t,r)\geq q_t(r_t)=:\delta_t>0\;\;\forall r\geq r_t,
\end{equation*}
implying that
\begin{equation*}
\int_{\R^3}|\nabla m|(t,x)\de\mu(x)=4\pi\int_0^\infty q_t(r)\de r\geq 4\pi\int_{r_t}^\infty\delta_t\de r=+\infty.
\end{equation*}
\end{proof}

\begin{proof}[Proof of Theorem \ref{thm:equiv_failure}]
    Consider the weighted manifold $(M,g,\mu)$ of Proposition \ref{prop:grad_heat_blowup}. Let $E$ be a bounded set of finite perimeter and assume by contradiction that there exists $t>0$ such that \eqref{eq:grad_heat_blowup} is finite. For such $t$ we would have 
    \begin{equation*}
|\nabla{\rm h}_t\chi_E|+|\nabla{\rm h}_t\chi_{E^c}|\geq\bigg|\nabla\int_Mp_t(\cdot,y)\de\mu(y)\bigg| .
    \end{equation*}
    Integrating over $M$, exploiting that $\|\nabla {\rm h}_t\chi_E\|_{L^1(M)}$ is finite for all $t>0$ and \eqref{eq:grad_heat_blowup} we then reach a contradiction.
\end{proof}
\begin{remark}
\label{rem:counter_manifold}
    One can obtain the same pathological behavior with a conical manifold $(M,g)$ which is smooth everyhere except at a point. Indeed it suffices to consider $M=[0,\infty)\times\mathbb{S}^2$ with $g=\de r^2+\psi^2(r)g_{\mathbb{S}^2}$, where $\psi(r)=re^{\frac{r^4}{2}}$. It is easy to check that the Laplace-Beltrami operator $\Delta_g$ on radial functions is the same as the one of the weighted manifold example. Moreover the integral of the heat kernel $m(t,x)$ is still a radial function and the rest of the proof translates verbatim to this setting.
\end{remark}
\noindent \textbf{Acknowledgements.} 
The author would like to thank Michele Caselli and Diego Pallara for a careful reading of a preliminary version of this manuscript and Diego Pallara for proposing the problem. \\
This work was supported by UK Research and Innovation (UKRI) under the Horizon Europe funding guarantee [grant number EP/Z000297/1].
\bibliography{references}
            \bibliographystyle{alpha}

\end{document}